\documentclass{article}

\usepackage[utf8]{inputenc}

\usepackage{amsmath}
\usepackage{euscript}
\usepackage{amssymb}
\usepackage{amsfonts}
\usepackage{amsthm}
\usepackage{amsopn}
\usepackage{ stmaryrd }

\theoremstyle{definition}
\newtheorem{theorem}{Theorem}
\newtheorem{definition}[theorem]{Definition}
\newtheorem{corollary}[theorem]{Corollary}
\newtheorem{lemma}[theorem]{Lemma}

\numberwithin{theorem}{section}

\newcommand{\F}{\mathbb{F}}
\newcommand{\X}{\mathbf{X}}
\newcommand{\Pp}{\mathbb{P}}

\newcommand{\Glm}{\mathcal{G}_{\ell,m}}

\newcommand{\Mts}{\mathbb{M}^{\ell \times m}(\F_{q^2})}
\newcommand{\Df}{\Delta(\ell, m)}
\newcommand{\Dfts}{\Delta(3, 6)}

\newcommand{\CO}{C(\mathbb{O}_{3,6})}

\newcommand{\Oll}{\mathbb{O}_{\ell, 2\ell}}
\newcommand{\Ots}{\mathbb{O}_{3,6}}
\title{Computing the minimum distance of the $C(\mathbb{O}_{3,6})$ polar Orthogonal Grassmann code with elementary methods}
\author{Sarah Gregory, Fernando Pi\~nero--Gonz\'alez, \\ Doel Rivera--Laboy, Lani Southern}
\begin{document}
\maketitle
\begin{abstract}
    The polar orthogonal Grassmann code $\CO$ is the linear code associated to the Grassmann embedding of the Dual Polar space of $Q^+(5,q)$. In this manuscript we study the minimum distance of this embedding. We prove that the minimum distance of the polar orthogonal Grassmann code $\CO$ is $q^3-q^3$ for $q$ odd and $q^3$ for $q$ even. Our technique is based on partitioning the orthogonal space into different sets such that on each partition the code $\CO$ is identified with evaluations of determinants  of skew--symmetric matrices. Our bounds come from elementary algebraic methods counting the zeroes of particular classes of polynomials. We expect our techniques may be applied to other polar Grassmann codes.
    
\end{abstract}

\section{Introduction}
Let $q$ be a prime power and $\F_q$ denote the finite field of $q$ elements. The Grassmannian, $\mathcal{G}_{\ell, m}$, is the collection of all subspaces of dimension $\ell$ of a vector space $V$ of dimension $m$. We take $V=\F_q^m.$ This is a highly interesting and well studied geometry with a rich algebraic structure. 

It is well known that the Grassmannian $\mathcal{G}_{\ell, m}$ may be embedded into the projective space $\Pp(\binom{m}{\ell}-1, \F_q) $ through the Pl\"ucker embedding, which is described as follows. For each vector space $W \in \mathcal{G}_{\ell, m}$, pick an $\ell \times m$ matrix $M_W$   such that the rowspace of $M_W$ is $W$. Let $I_1, I_2, \ldots, I_{\binom{m}{\ell}}$ denote all the subsets of the set $[m]$ with  $\ell$ elements.  Then $M_W$ is mapped to the projective point $$p_W := \left(det_{I_1}M_W : det_{I_2}M_W : \cdots : det_{I_{\binom{m}{\ell}}}M_W\right),$$ where the i-th position is the value of the  determinant of the submatrix of $M_W$ with the columns in the set $I_i$  (i.e. the $\ell$ minor of $M_W$ defined by the columns of $I_i$). The map $W \mapsto p_W$ is well--defined up to projective equivalence. This map is also the embedding of the dual space defined by the point--line incidence relations of the Grassmannian in projective space.


Linear codes may be used to study the Grassmannian. The Grassmann code is defined as the linear code whose generator matrix is defined by taking all of the points $p_W$ as its column vectors. Grassmann codes were introduced for $q = 2$ in \cite{R1, R2} and for general $q$, in \cite{Nogin}, Nogin studied the parameters of this linear code, denoted $C(\ell, m)$. The minimum distance of the Grassmann code is of particular importance as it determines the maximum number of points a hyperplane has in common with the points of the Grassmannian.

Polar Grassmannians are special subvarieties of the Grassmannian. If $B: V\times V \rightarrow \F$ is a bilinear form, a polar Grassmannian is a subvariety of the Grassmannian where the subspaces $W \in \mathcal{G}_{\ell, m} $ satisfy the relation $B(x,y) = 0 \forall \ x,y \in W$. If the form $B$ is a quadratic form, the subvariety is known as an orthogonal Grassmannian and he subspaces satisfying $B(x,y) = 0$ are the nonsingular spaces. If the form $B$ is an alternant form, the subvariety is known as a symplectic Grassmannian and the subspaces satisfying $B(x,y) = 0$ are the totally isotropic spaces. If $m = 2\ell$, then the symplectic Grassmannian is also known as the Lagrangian Grassmannian. If the form $B: V \times V \rightarrow \F$ is a sesquilinear form, the subvariety $\{W \in \Glm \ : \ B(v,w) = 0 \ \forall v,w \in W \}$  is known as a unitary Grassmannian or a Hermitian Grassmannian and the subspaces satisfying $B(x,y) = 0$ are the totally isotropic spaces. These subvarieties also define dual polar spaces with a defined incidence geometry of points and lines. The Pl\"ucker embedding is also used to embed the varieties (and the associated geometries) into the appropriate projective spaces. 

The dimension of these embeddings when the bilinear forms are nondegenerate were determined by Blok and Cooperstein \cite{blok2010generating} and for more general forms (with general Witt index) by Cardinali, Giuzzi and Pasini \cite{PolarSurvey}. Much less is known about the minimum distance of these codes. In particular the distance is known when $\ell = 2$ (i.e. the Grassmannian of lines) and some cases where $\ell = 3$.  Cardinali and Giuzzi \cite{Symplectic} determined the minimum distance of polar Grassmann codes under a symplectic form (polar symplectic Grassmannian) , under an orthogonal form (polar orthogonal Grassmannian) \cite{cardinali2014line, Cardinali_2018} and under a Hermitian form (polar Hermitian Grassmannian) \cite{Cardinali} for $\ell = 2$ and all $m$. In \cite{Symplectic} for $\ell = 3 $ and $m = 6$ Cardinali and Giuzzi state the full weight enumerator of the corresponding symplectic Grassmann code. Cardinali and Giuzzi in \cite{CG18} also determine a way to encode the elements of the polar Hermitian space as a way to aid in encoding and decoding codes from the polar Hermitian Grassmannian. 


Cooperstein \cite{COOPERSTEIN1997849} and de Bruyn \cite{DB10}  have also studied the polar Hermitian Grassmannian as a dual polar space. Cooperstein proved that for $m = 2\ell$ the polar Hermitian Grassmannian may be embedded in $\Pp(\binom{m}{\ell}-1, \F_{q})$.  De Bruyn and Pralle \cite{DP08} characterized the hyperplanes of the dual polar space $DH(5,q^2)$. The hyperplanes of that dual polar space are in $1$--$1$ correspondence with the codewords of $C(\mathbb{H}_{3,6})$. Thus the weight distributuon of $C(\mathbb{H}_{3,6})$ is known. The benefit of our technique is that it may be readily generalized to  $C(\mathbb{H}_{4,8})$. In contrast, the techniques used in \cite{DP08} do not generalize to the case $DH(7,q^2)$. We also apply this technique to determining the minimum distance of the polar Orthogonal Grassmann code $\CO$.


\section{Preliminaries}

In this manuscript $q$ denotes a prime power, $\F_q$ denotes the field with $q$ elements and $\F_{q^2}$ denotes the field with $q^2$ elements. The collection of all $\ell\times m$ matrices over $\F_{q^2}$ is denoted by $\Mts$.

\begin{definition}
Let $V$ be a vector space of dimension $m$ over the field $\F$. The Grassmannian $\Glm(V)$ over $\F$ is the set of all $\ell$-dimensional subspaces of $V$.
\end{definition}

Since $V$ is an $m$--dimensional vector space over $\F$, $V$ is isomorphic to $\F^m$. When the choice of $V$ is clear we may omit $V$ from the notation and denote the Grassmannian by $\Glm$.

\begin{definition}
Let $m$ be an integer. We shall denote by $J$ the following $m \times m$ matrix whose entries are $0$ or $1$.

$$ J_{i,j} := 
  \begin{cases} 
       1 & \makebox { if } j = m+1-i \\
       0 & \makebox{ if } j \neq m+1-i
   \end{cases}
$$
\end{definition}

\begin{definition}

Let $V$ be a vector space over $\F$.  Suppose that $B: V \times V \rightarrow \F$ satisfies 
\begin{itemize}
    \item $B(x+x', y) =  B(x,y) + B(x',y)$
    \item $B(x, y+y') =  B(x,y) + B(x,y')$
    \item $B(\alpha x, y) = \alpha B(x,y)$
    \item $B(x, \alpha y) = \alpha B(x,y)$
    
\end{itemize}
for all $\alpha \in F, x,y \in V$. We say that $B$ is a \emph{symmetric form}.
\end{definition}

In our work we shall use the symmetric form given by $B(x,y) := xJ y^{T}$ with associated quadric $x_1x_{2m} + x_{2}x_{2m-1} + \cdots + x_{m}x_{m+1}$. This Hermitian form is different from the one used by Cardinali and Giuzzi \cite{Cardinali, CG18} and by de Bruyn \cite{DB10}. As in the case of the Hermitian form $J$, we can neatly partition the Grassmann subvariety into Schubert cells. These Schubert cells are represented by subspaces of skew--symmetric matrices. In the case of even characteristic the skew--symmetric matrices will have zeroes on the diagonals. We bound the number of zeroes of particular classes polynomials using elementary techniques to establish bounds on the minimum distance of the polar orthogonal Grassmann code.

\begin{definition}
Let $V$ be a vector space of dimension $m$ over the field $\F$. Let $Q$ a quadratic form and let $B$ be its associated bilinear form. Let $W \in \Glm$ where $m = 2\ell$. We say $W$ is \emph{totally singular} with respect to the Hermitian form $B$ if $$Q(v)=0,Q(w) =0,B(v,w) = 0,
\ \forall v,w \in W.$$


The $\ell$--space $W$ is totally singular with respect to $J$ if and only if for all $x,y \in W$ we have that $$\sum\limits_{i=1}^{2\ell} x_iy_{2\ell+1-i}= 0 $$
\end{definition}

\begin{definition}
Let $V$ be a vector space of dimension $m$ over the field $\F$. Let $B$ be a symmetric form and let $Q$ be its associated quadric. The \emph{polar orthogonal Grassmannian } $\Oll$ is defined as the set of totally singular spaces of $\Glm.$ That is $$\Oll := \{W \in \Glm \ : \ B(v,w) = 0, Q(v) = 0 \forall v,w \in W  \}.$$
\end{definition}



\section{Polar orthogonal Grassmann Code}

In this section we state the definition of polar orthogonal Grassmann code.

Recall that $det_{\X}$ represents the minor of the generic matrix $\X$ on the columns given by $I$. In this section we change the notation of the space of minors as follows.
\begin{definition}


We denote by $\Df$ the set of all $\F_{q}$--linear combinations of the $\ell$--minors of the generic $\ell \times m$ matrix $\X$, the functions $det_I(\X)$. That is $$\Df := \{ \sum c_I det_I(\X) \ : \ c_I \in \F_{q} \} $$
\end{definition}

\begin{definition}
Let $f\in \Df$, We define the \emph{support} of $f$ as the collection of all column sets whose minors appear in the expression of $f$ with a nonzero coefficient. That is

 $$supp(f) = \{ I \subseteq [m] \ \| \ c_I \neq 0, f = \sum c_I det_I(\X) \}.$$
\end{definition}

 Now we define an evaluation map from the linear combinations of determinants.

\begin{definition}

Let $L = \{M_1, M_2, \ldots, M_n\} \subseteq \Mts$ be a subset of matrices. Let $f \in \Df$. We define the evaluation map of $f$ onto $L$ as: $$ev_L(f) = (f(M_1), f(M_2), \ldots, f(M_n)).$$ 

\end{definition}









\begin{definition}
For each $W \in \Oll$ let $M_W$ denote a $\ell \times 2\ell$ matrix whose rowspace is $W$. Suppose $L$ is a collection of matrix representatives for each elements of $\Oll$ where each $W\in \Oll$ has a unique representative in $L$. The \emph{polar Orthogonal Grassmann code}, $\CO$, is defined as $$ Im(ev) = \{ ev_L(f) \ : \ f \in \Df \}. $$

\end{definition}







Since the Pl\"ucker embedding of $\Oll$ into the projective space $\Pp(\binom{2\ell}{\ell}-1,q)$ is does not map lines to lines calculating the minimum distance of the code $\CO$ is not equivalent to determining the maximum possible number of points $\Oll$ has in common with an algebraic hyperplane. Nonetheless, it is worthwhile to understand the Grassmannian embedding in this case, as it is still a pointwise embedding. We'll determe the minimum distance of the polar orthogonal Grassmann code for $\ell = 3$ by evaluating the set of all $3$--minors of a generic $3 \times 6$ matrix on the set of all matrix representatives of $\Ots$, restricting $f$ to special subsets of $\Ots$ and determining the weight of $f$ on each of these subsets.





\section{Automorphisms of $\Df$ in the orthogonal case}

Now we have to ensure that the column operations on $\X$ preserve the orthogonality constrain. For simplicity we shall use $B(x,y)$ to denote $$B(x,y) =  \sum\limits_{i=1}^{2\ell} x_iy_{2\ell+1-i}$$ and $Q(x)$ to denote the associated quadratic form $$Q(x) = \sum\limits_{i=1}^{\ell} x_ix_{2\ell+1-i} $$    




\begin{lemma}\label{lem:orthogonal1}

Let $x = (x_1, x_2, \ldots, x_{2\ell})$ and $y = (y_1, y_2, \ldots,y_{2\ell})$. Let $\sigma$ be the map obtained by applying the following two column operations $C_i + aC_j \rightarrow C_i$ and $C_{2\ell+1-j} -aC_{2\ell+1-i} \rightarrow C_{2\ell+1-j}$ where $i \neq 2\ell+1-j.$ Suppose that $B(x,y) = 0$ and $Q(x)  = 0$. Then $$B(\sigma(x), \sigma(y)) = 0, Q(\sigma(x)) = 0. $$
\end{lemma}
\begin{proof}

Let $\sigma(x) = v = (v_1, v_2, \ldots, v_{2\ell})$ and $\sigma(y) =w = (w_1, w_2, \ldots, w_{2\ell})$.

$$B(v, w) = \sum\limits_{r=1}^{2\ell} v_rw_{2\ell+1-r}.$$
We write out the terms from the sum $B(v, w)$ with indices $i,j,2\ell+1-i, 2\ell+1-j$ and write

\begin{equation} \label{eq1}
\begin{split}
 B(v, w) & = v_{i}w_{2\ell+1-i} + v_{2\ell+1-i}w_{i}+ v_{j}w_{2\ell+1-j} + \\ 
 & v_{2\ell+1-j}w_{j} + \sum\limits_{r \neq i,j,2\ell+1-i,2\ell+1-j} v_rw_{2\ell+1-r}
\end{split}
\end{equation}

Writing everything in terms of $x_i$ and $y_i$ we obtain:

\begin{tabular}{rcl}
  $B(v, w)$   &  $=$ & \\
             &  &  $(x_i + ax_j)y_{2\ell+1-i} + x_{2\ell+1-i}(y_i + ay_j)$\\
             & $+$ & $x_{j}(y_{2\ell+1-j} - ay_{2\ell+1-i}) + (x_{2\ell+1-j} - ax_{2\ell+1-i})y_{j}$ \\
     & $+$ & $\sum\limits_{r \neq i,j,2\ell+1-i,2\ell+1-j} x_r y_{2\ell+1-r}.$
\end{tabular}

Expanding the sums we obtain:

\begin{tabular}{rcl}
  $B(v, w)$   &  $=$ & \\
             &  &  $x_iy_{2\ell+1-i} + ax_jy_{2\ell+1-i} + x_{2\ell+1-i}y_i + x_{2\ell+1-i}ay_j$\\
             & $+$ & $x_{j}y_{2\ell+1-j} -x_{j}ay_{2\ell+1-i} + x_{2\ell+1-j}y_j - ax_{2\ell+1-_i}y_{j}$ \\
     & $+$ & $\sum\limits_{r \neq i,j,2\ell+1-i,2\ell+1-j} x_r y_{2\ell+1-r}.$
\end{tabular}


We rearrange some terms as:

$$B(v, w) = ax_jy_{2\ell+1-i} + x_{2\ell+1-i}ay_j -x_{j}ay_{2\ell+1-i} - ax_{2\ell+1-_i}y_{j} + \sum\limits_{r =1}^{2\ell} x_r y_{2\ell+1-r}.$$

The terms $ax_jy_{2\ell+1-i} + x_{2\ell+1-i}ay_j -x_{j}ay_{2\ell+1-i} - ax_{2\ell+1-_i}y_{j}$ cancel and we obtain:

$$B(v, w) = B(x,y) = 0.$$
Now we consider the quadratic form $Q$.
$$Q(v) = \sum\limits_{r=1}^{\ell} v_rv_{2\ell+1-r}.$$
We write out the terms from $Q(v)$ with indices $i,j,2\ell+1-i, 2\ell+1-j$ and write

\begin{equation} \label{eq3}
\begin{split}
 Q(v) & = v_{i}v_{2\ell+1-i} + v_{2\ell+1-i}v_{i}+ v_{j}v_{2\ell+1-j} + \\ 
 & v_{2\ell+1-j}v_{j} + \sum\limits_{r \neq i,j,2\ell+1-i,2\ell+1-j} v_rv_{2\ell+1-r}
\end{split}
\end{equation}

Writing everything in terms of $x_i$ we obtain:

\begin{tabular}{rcl}
  $Q(v)$   &  $=$ & \\
             &  &  $(x_i + ax_j)x_{2\ell+1-i} + x_{2\ell+1-i}(x_i + ax_j)$\\
             & $+$ & $x_{j}(x_{2\ell+1-j} - ax_{2\ell+1-i}) + (x_{2\ell+1-j} - ax_{2\ell+1-i})x_{j}$ \\
     & $+$ & $\sum\limits_{r \neq i,j,2\ell+1-i,2\ell+1-j} x_r x_{2\ell+1-r}.$
\end{tabular}

Expanding the sums we obtain:

\begin{tabular}{rcl}
  $Q(v)$   &  $=$ & \\
             &  &  $x_ix_{2\ell+1-i} + ax_jx_{2\ell+1-i} + x_{2\ell+1-i}x_i + x_{2\ell+1-i}ax_j$\\
             & $+$ & $x_{j}x_{2\ell+1-j} -x_{j}ax_{2\ell+1-i} + x_{2\ell+1-j}x_j - ax_{2\ell+1-_i}x_{j}$ \\
     & $+$ & $\sum\limits_{r \neq i,j,2\ell+1-i,2\ell+1-j} x_r x_{2\ell+1-r}.$
\end{tabular}


We rearrange some terms as:

$$Q(v) = ax_jx_{2\ell+1-i} + x_{2\ell+1-i}ax_j -x_{j}ax_{2\ell+1-i} - ax_{2\ell+1-_i}y_{j} + \sum\limits_{r =1}^{2\ell} x_r y_{2\ell+1-r}.$$

The terms $ax_jx_{2\ell+1-i} + x_{2\ell+1-i}ax_j -x_{j}ax_{2\ell+1-i} - ax_{2\ell+1-_i}x_{j}$ cancel and we obtain:

$$Q(v) = Q(x) = 0.$$\end{proof}












We remark that any column permutation which permutes the first $\ell$ columns with in any way and the last $\ell$ columns in an analogous way will preserve singularity. Given $\eta \in S_\ell$, define i.e. $\sigma \in S_{2\ell}$ where $\sigma(i) = \eta(j)$ for both $1 \leq i,j \leq\ell $ and for $\ell +1 \leq i,j \leq 2\ell$ $\sigma(i) = 2\ell+1-\eta(2\ell+1-j)$.

 \section{Schubert cell partition of $\Oll$}
 
 \begin{lemma} Let $W\in \Oll$. Let $\ell \geq 2$. Let $A$ be an $\ell \times 2\ell$ matrix whose rowspace is $W$. Suppose $I$ is the set of $\ell$ pivots of $A$. For all $1 \leq i \leq 2\ell$, $I$ cannot contain both $i \in I$ and $2\ell+1-i \in I$.\end{lemma}
\begin{proof}
Suppose that $A$ is a matrix which satisfies $rowspace(A) \in \Oll$.  By way of contradiction we shall assume that $A$ is an $\ell \times 2\ell$ matrix with $\ell \geq 2$, rank at least $2$ and two of its pivots in columns $x$ and $2\ell-x+1.$ Let $r_a=(a_1,a_2,...a_{2\ell})$ be the row with pivot in column $x$ and $r_b=(b_1,b_2,...b_{2\ell})$ be the row with pivot in column $2\ell-x+1.$ The product $H(r_b, r_a)$  is given by 
$$B(r_b, r_a) = \sum_{i=1}^{2\ell} b_{m-i+1}{a_i}.$$

The conditions on the rows $r_a$, $r_b$ and the pivots of $A$ force $a_i=0$ for all $i>x$ and $b_{2\ell-i+1}=0$ for all $2\ell-i+1>2\ell-x+1.$ When $i>x,$ $a_i=0$ so $b_{2\ell-i+1}{a_i}=0.$ When $i<x,$ $2\ell-i+1>2\ell-x+1.$ Then $b_{2\ell-i+1}=0$ and $b_{2\ell-i+1}{a_i}=0.$ When $i=x$ both $a_i$ and $b_{2\ell-i+1}$ are pivots so $a_i={a_i}=b_{2\ell-i+1}=1$ and $b_{2\ell-i+1}{a_i}=1.$ Therefore each term of the sum is equal to $0$ except when $i=x$ and then the term equals $1.$

Thus we have $$B(r_b, r_a) = \sum_{i=1}^{2\ell} b_{2\ell-i+1}{a_i} = b_{2\ell-x+1}a_x=1 \not= 0$$ and the space represented by $A$ is not singular. \end{proof}

The only possible pivot sets of a singular space are the sets $123$, $124$, $135$, $145$, $236$, $246$, $356$ and $456$. If we consider pivot sets when reducing rows right--to--left (instead of left--to--right) we still end up with the same possible pivot sets. 

One advantage to using the symmetric form given by $J$ and the quadric $x_1x_6+x_2x_5+x_3x_4$ is that we may partition the matrices in $L$ representing the totally isotropic spaces of $\Ots$ into the following sets of matrix representatives.

\begin{itemize}
\item $P_{456}= \Bigg\{ \begin{bmatrix}
0 & a_2 & a_3 & 0 & 0 & 1\\
-a_2 & 0 & a_5 & 0 & 1 & 0\\
-a_3 & -a_5 & 0 & 1 & 0 & 0\\
\end{bmatrix} \ : \ a_2,a_3, a_5 \in \F_{q} \Bigg\}  $

\item $P_{356}= \Bigg\{  \begin{bmatrix}
0 & b_2 & 0 & b_3 & 0 & 1\\
-b_2 & 0 & 0 & b_5 & 1 & 0\\
-b_3 & -b_5 & 1 & 0 & 0 & 0\\
\end{bmatrix} \ : \ b_2, b_3, b_5 \in \F_{q} \Bigg\} $


\item $P_{246}= \Bigg\{ \begin{bmatrix}
0 & 0 & c_2 & 0 & c_3 & 1\\
-c_2 & 0 & 0 & 1 & 0 & 0\\
-c_3 & 1 & 0 & 0 & 0 & 0\\
\end{bmatrix} : c_2, c_3 \in \F_{q} \Bigg\}  $


\item $P_{236}= \Bigg\{  \begin{bmatrix}
0 & 0 & 0 & d_2 & d_3 & 1\\
-d_2 & 0 & 1 & 0 & 0 & 0\\
-d_3 & 1 & 0 & 0 & 0 & 0\\
\end{bmatrix} \ : \ d_2, d_3 \in \F_{q} \Bigg\}$


\item $P_{145}=\Bigg\{  \begin{bmatrix}
0 & 0 & e_2 & 0 & 1 & 0\\
0 & -e_2 & 0 & 1 & 0 & 0\\
1 & 0 & 0 & 0 & 0 & 0\\
\end{bmatrix} \ : \ e_2 \in \F_{q} \Bigg\}$


\item $P_{135}= \Bigg\{ \begin{bmatrix}
0 & 0 & 0 & x_2 & 1 & 0\\
0 & -x_2 & 1 & 0 & 0 & 0\\
1 & 0 & 0 & 0 & 0 & 0\\
\end{bmatrix} \ : \ x_2 \in \F_{q} \Bigg\}$ 


\item $P_{124}=\Bigg\{ \begin{bmatrix}
0 & 0 & 0 & 1 & 0 & 0\\
0 & 1 & 0 & 0 & 0 & 0\\
1 & 0 & 0 & 0 & 0 & 0\\
\end{bmatrix} \Bigg\}$


\item $P_{123}= \Bigg\{ \begin{bmatrix}
0 & 0 & 1 & 0 & 0 & 0\\
0 & 1 & 0 & 0 & 0 & 0\\
1 & 0 & 0 & 0 & 0 & 0\\
\end{bmatrix} \Bigg\}$
\end{itemize}

These are the possible patterns of the matrices when performing row reduction from right to left. We have chosen this representation because the submatrix in the non pivot positions is a skew--symmetric matrix with zeroes in the diagonal.  Evaluating the functions $f \in \Dfts$ on the matrices in each $P_I$ is the same as evaluating all linear combinations of minors of $3 \times 3$ skew--symmetric matrices which may have some special zero pattern.

\begin{lemma}
Let $I$ be a pivot set such that $4 \in I$. Let $I' = I \cup \{3\} \setminus \{4\}$. Then the matrix representatives in $P_{I'}$ may be obtained by swapping columns $3$ and $4$ from the matrix representatives in $P_I$.
\end{lemma}

\section{Properties of minors on $\Ots$}

In this section we prove certain properties of functions of the form $det_{A}(X)$ evaluated on the different partition sets of $\Ots.$

\begin{definition}

Given $I \subseteq [2\ell]$ define $$I^* := \{ 2\ell+1 - i' \ : \ i \in [2\ell], i \not \in I \}.$$
\end{definition}
Our aim is to relate the functions $det_{A^*}(\X)$ and $det_{A}(\X)$ when evaluated on each set $P_I$.

\begin{definition}

Let $I$ be a set of columns satisfying $I = I^*$. Let $A$ be any set of $\ell$ columns. Define the sets $P,N,r_{A,I}, c_{A,I}$ as follows:
\begin{itemize}
    \item  $I := \{ i_1 < i_2 < \ldots < i_\ell \}$
    \item  $N  := [2\ell] \setminus I = \{ n_1 < n_2 < \ldots n_\ell \}$
    \item $r_{A,I} := \{ r \ : \ i_{\ell+1-r} \in I \setminus A \} $
    \item $c_{A,I} := \{ s \  : \ n_s \in N \cap A \}$
\end{itemize}

\end{definition}

Now we relate the $3 \times 3$ minors of the $ 3 \times 6$ matrices in $P_I$ with the minors of the $3 \times 3$ nonidentity submatrix. (i.e. the submatrix whose columns are not in $I$)

\begin{lemma}
Suppose $I \subseteq [2\ell]$ such that $I^* = I$. Let $P_I$ be the set of matrices with pivot set $I$. Let $A \subseteq [m]$. Let $M$ be a $\ell \times m$ matrix such that $M \in P_I$. Evaluating the minor $det_{A}(M)$ on $P_I$ is equivalent to evaluating $ (-1)^{\sum_{i_r=a_j \in A} \ell -r+j+1} det_{r_{A,I}, c_{A,I}}(M_{I^C})$.

\end{lemma}

\begin{proof}
Let $I = \{i_1, i_2, \ldots, i_\ell\}$. Note that these columns are of the form:
    $$m_{r,i_j} = \begin{cases}
         1 & r = \ell+1-j\\
         0 & else\\
    \end{cases}$$
Hence, expanding $det_A(M)$ along $r \in P\cap A$, we get $$det_A(M) = (-1)^{\ell-r+1+j}det(M^{\{[\ell]- \{\ell+1-i\} \}, \{A-\{I\} \}}).$$

The determinant $det_{A}$ is specified by the columns not in common between $I$ and $A$. Note that when removing the columns in $I$, the order is preserved for the other columns. Thus we may obtain the columns in $M_{I^C}$ by taking the subscripts of the ordered set $N = I^C = \{n_1, n_2,\ldots, n_{\ell}\}$. 

For the rows, note that each column $i_r \in I$ removes the row $\ell+1-i$ from consideration within the evaluation of the minor. This implies that the rows remaining in the evaluation are those rows in $I$ not evaluated in $A$. Thus the rows evaluated are $\{r : i_{\ell+1-r}\in I\setminus A \}.$ 

Expanding along the columns in $I \cap A$, we obtain that $det_A(M)$ is given by the columns in 
$\{ s \  : \ n_s \in N \cap A \}$ and the rows in $ \{ r \ : \ i_{\ell+1-r} \in I \setminus A \}$. Therefore  $$ det_A(M) = (-1)^{\sum_{i_r=a_j \in A} \ell -r+j+1} det_{r_{A,I}, c_{A,I}}(M_{I^C}).$$ \end{proof}

\begin{definition}
We say a minor $A$ is \emph{principal} along a pivot set $I$ if $r_{A,I} = c_{A,I}$. Consequently, if $r_{A,I} \neq c_{A,I}$
then the minor is \emph{non-principal}.
\end{definition}

\begin{lemma}\label{lem:determinanttranspose}
Let $A \subseteq [2\ell]$ where $\# A  = \ell$. Let $B = A^C$. Let $I = \{ i_1 < i_2 < \ldots i_\ell \}$
and $N = I^C = \{ n_1 < n_2 < \ldots n_\ell \}$ where $I = N^*$. Then $r_{A,I} = c_{B,I}.$
\end{lemma}
\begin{proof}
 
Suppose that $s \in c_{B,I}.$ By definition $n_s \in N \cap I$. By the definition of $N$ and $B$ we have that $n_s \in N^* \cap A^*$. This implies that $m+1-n_s \in I \cap A$. From the definition of $N$ and the fact that $N = I^*$ we have that $i_{\ell+1-s} \in I \cap A$. Therefore $s \in r_{A,I}$ \end{proof}

\begin{corollary}\label{coll: submat conj} Let $q$ be even and $P_I$ be the set of matrix representatives with pivot set $I$ in $\Oll$. Then the evaluation of $det_{A^*}(\X)$ on $P_I$ is equal to $det_{A}(\X)$. 
\end{corollary}
\begin{proof}
Let $I$ be a pivot set of $\Oll$. Let $M(P_I)$ be the set of matrices in $P_I$ whose columns are restricted to those columns not in $I$. Then $ev_{P_I}(det_A(\X)) = det_{r_{A,P}, c_{A,P}}(M(P_I))$. Lemma \ref{lem:determinanttranspose} implies $ev_{P_I}(det_{A^*}(\X)) = det_{{c_{A,P}, r_{A,P}}}(M(P_I))$. As the matrix $M(P_I)$ is a skew--symmetric matrix, the determinants satisfy  $det_{c_{A,I}, r_{A,I}}(M(P_I)) = (-1)^{\# c_A} det_{r_{A,I}, c_{A,I}}(M(P_I))$. Since we are in even charcterisstic and $\# c_A = \# (A \setminus I)$, the statement follows. \end{proof}

We end this section proving we may assume certain minors appear in the support of $f$ without loss of generality.

\begin{lemma}\label{lem:principal123}
Let $f \in \Dfts$. If there is a principal minor in the support of $f$, then there exists $g \in \Dfts$ such that $123 \in supp(g)$ and $wt(f) = wt(g)$.

\end{lemma}
\begin{proof}
If $124 \in supp(f)$ then $g = f(X_1 \| X_2 \| X_4 \| X_3 \| X_5 \| X_6)$ satisfies $123 \in supp(g)$ and $wt(g) = wt(f)$. The sets $135, 236 \in supp(f)$ are similar.

If $145 \in supp(f)$ then $g = f(X_1 \| X_5 \| X_4 \| X_3 \| X_2 \| X_6)$ satisfies $123 \in supp(g)$ and $wt(g) = wt(f)$. The sets $246$,$356$ are similar.

If $456 \in supp(f)$ then $g = f(X_6 \| X_5 \| X_4 \| X_3 \| X_2 \| X_1)$ satisfies $123 \in supp(g)$ and $wt(g) = wt(f)$. 

\end{proof}

\begin{lemma}\label{lem:principal125}
Let $f \in \Dfts$. If there is a nonprincipal minor in the support of $f$, then there exists $g \in \Dfts$ such that $125 \in supp(g)$ and $wt(f) = wt(g)$.

\end{lemma}
\begin{proof}
Let  $I \in supp(f)$, such that $I \neq I^*$. If $3,4 \in I$, then $$h = f(X_1 \| X_3 \| X_2 \| X_5 \| X_4 \| X_6).$$
If $1,6 \in I$, then $$h = f(X_2 \| X_1 \| X_3 \| X_4 \| X_6 \| X_5).$$

Note there is $J$ such that $J \in supp(h)$ such that $2,5 \in J$.
If $1 \in J$, then $g = h$ satisfies $125 \in supp(g)$ and $wt(g) = wt(f)$.
If $6 \in J$ then $g = h(X_6 \| X_2 \| X_3 \| X_4 \| X_5 \| X_1)$ satisfies $125 \in supp(g)$ and $wt(g) = wt(f)$.
If $3 \in J$ then $g = h(X_3 \| X_2 \| X_1 \| X_6 \| X_5 \| X_4)$ satisfies $125 \in supp(g)$ and $wt(g) = wt(f)$.
If $4 \in J$ then $g = h(X_3 \| X_2 \| X_6 \| X_1 \| X_5 \| X_4)$ satisfies $125 \in supp(g)$ and $wt(g) = wt(f)$.

\end{proof}

\section{Calculating the minimum distance of $\CO$}

Having partitioned $\Ots$ according to the pivots of their representative matrices, we can now calculate the weight of a codeword by computing its weight on each individual pivot set. We split the proof in three cases: $q$ odd, $q=2$ and $q$ even but $q > 2$.  We begin with the odd characteristic case.
\subsection{odd characteristic case}

In this subsection we assume $q$ is an odd prime power. We begin by computing $wt(det_{236}(\X)-det_{456}(\X))$.

\begin{lemma}\label{lem:min codew-orthogonal-odd}
$$wt(det_{236}(\X)-det_{456}(\X)) = q^3-q^2. $$
\end{lemma}
\begin{proof}

Let $f = det_{236}(\X)-det_{456}(\X)$. Recall that the matrices in $P_{456}$ are of the form $\begin{bmatrix}
0 & a_2 & a_3 & 0 & 0 & 1\\
-a_2 & 0 & a_5 & 0 & 1 & 0\\
-a_3 & -a_5 & 0 & 1 & 0 & 0\\
\end{bmatrix}$ where $a_2,a_3,a_5 \in \F_{q}$. Note that for $M_W \in P_{456}$, $det_{456}(M_W) = 1$ and $det_{236}(M_W) = \begin{vmatrix}  a_2& a_3 & 1 \\
 0 & a_5 & 0 \\
 -a_5 & 1 & 0 \\ \end{vmatrix} = a_5^2.$ Therefore $f(M_W) = a_5^2-1$. For any of the $q^2$ possible values of $a_2$ and $a_3$, $f$ is determined by the values of $a_5$. We note that $f(A) = 0$ if $a_5=1$ or $a_5 = -1$. There are $q-2$ values such that $f(A) \neq 0$. Therefore $wt_{P_{456}}(f) = (q-2)q^2$.

Recall that $P_{356}$ is of the form $  \begin{bmatrix}
0 & b_2 & 0 & b_3 & 0 & 1\\
-b_2 & 0 & 0 & b_5 & 1 & 0\\
-b_3 & -b_5 & 1 & 0 & 0 & 0\\
\end{bmatrix}$ where the entries  $b_2, b_3, b_5 \in \F_{q} $. If $M_W \in P_{356}$ then $det_{456}(M_W) = 0$ since it has a zero row and $det_{236}(M_W) = \begin{vmatrix}  b_2& 0 & 1 \\
 0 & 0 & 0 \\
 -b_5 & 1 & 0 \\ \end{vmatrix} = 0.$  Therefore $wt_{P_{356}}(f) = 0$.

The matrices in $P_{246}$ are of the form $\begin{bmatrix}
0 & 0 & c_2 & 0 & c_3 & 1\\
-c_2 & 0 & 0 & 1 & 0 & 0\\
-c_3 & 1 & 0 & 0 & 0 & 0\\
\end{bmatrix}$ where $c_2, c_3 \in \F_q$. For any $M_W \in P_{246}$, $det_{236}(M_W) = \begin{vmatrix}  0& c_2 & 1 \\
 0 & 0 & 0 \\
 1 & 0 & 0 \\ \end{vmatrix} = 0.$ Therefore $f$ evaluates to $0$ on $P_{246}$.

The matrices in $P_{236}$ are of the form $\begin{bmatrix}
0 & 0 & 0 & d_2 & d_3 & 1\\
-d_2 & 0 & 1 & 0 & 0 & 0\\
-d_3 & 1 & 0 & 0 & 0 & 0\\
\end{bmatrix}$ where the entries  $d_2,d_3 \in \F_q$. For any $M_W \in P_{236}$, $det_{236}(M_W) = \begin{vmatrix}  0& 0 & 1 \\
 0 & 1 & 0 \\
 1 & 0 & 0 \\ \end{vmatrix} = 1.$ Therefore $f$ evaluates to $1$ on $P_{236}$ and $wt_{P_{236}}(f) = q^2$.
 
 Note that $236$ and $456$ contain column $6$. Column $6$ in each matrix in the the remaining pivot sets is zero. Therefore $det_{236}(M_W)$ and $det_{456}(M_W)$ evaluates to zero on the remaining pivot sets. Therefore $wt(f) = wt_{P_{456}}(f) + wt_{P_{236}}(f) = q^3-q^2$.  \end{proof}


We shall determine the weight of $f \in \Dfts$ when evaluated on the sets $P_I$. First we shall assume that $f$ contains no minor of the form $I = I^*$ in its support.

\begin{lemma}

Let $f \in \Dfts$ where $f \neq 0$. Suppose that no minor in the support of $f$ satisfies $I = I^*$. Then $wt(f) \geq q^3-q^2$
\end{lemma}
\begin{proof}
Let $f$ be as in the hypothesis of the Lemma. Since no term of $f$ is principal and $f \neq 0$  we may apply Lemma \ref{lem:principal125} and rescale. Without loss of generality we may assume $f_{125} = 1$. We bound $wt(f)$ by considering the contributions given by $det_{125}$ and $det_{134}$ on $P_{456}$ and $P_{356}$.
 
 The matrices in $P_{456}$ are of the form $ \begin{bmatrix}
0 & a_2 & a_3 & 0 & 0 & 1\\
-a_2 & 0 & a_5 & 0 & 1 & 0\\
-a_3 & -a_5 & 0 & 1 & 0 & 0\\
\end{bmatrix}$ where  $a_2, a_3, a_5 \in \F_q$. When we evaluate $f$ on the matrices $M_W$ note that no term of the form $a_2^2, a_3^2 , a_5^2$ appears because we have no minors satisfying $I = I^*$.

Since the determinant $det_{125}(M_W) = 
\begin{vmatrix}  0& a_2 & 0 \\
 -a_2 & 0 & 1 \\
 -a_3 & -a_5 & 0 \\ \end{vmatrix} = -a_2a_3$ and the determinant $det_{134}(M_W) = 
\begin{vmatrix}  0& a_3 & 0 \\
 -a_2 & a_5 & 0 \\
 -a_3 & 0 & 1 \\ \end{vmatrix} = a_2a_3$

Recall that $P_{356}$ is of the form $  \begin{bmatrix}
0 & b_2 & 0 & b_3 & 0 & 1\\
-b_2 & 0 & 0 & b_5 & 1 & 0\\
-b_3 & -b_5 & 1 & 0 & 0 & 0\\
\end{bmatrix}$ where $b_2, b_3, b_5 \in \F_q$. As in the case of $P_{456}$ when we evaluate $f$ on the matrices $M_W$ note that no term of the form $b_2^2, b_3^2 , b_5^2$ appears because we have no minors satisfying $I = I^*$. However, on this case $det_{125}(M_W) = 
\begin{vmatrix}  0& b_2 & 0 \\
 -b_2 & 0 & 1 \\
 -b_3 & -b_5 & 0 \\ \end{vmatrix} = -b_2b_3$ and $det_{134}(M_W) = 
\begin{vmatrix}  0& 0 & b_3 \\
 -b_2 & 0 & b_5 \\
 -b_3 & 1 & 0 \\ \end{vmatrix} = -b_2b_3$

Note that on $P_{456}$ $f$ evaluates to a polynomial on $a_2,a_3$ and $a_5$ where $a_2a_3$ appears with coefficient $f_{125}+f_{134}$. Note that on $P_{356}$ $f$ evaluates to a polynomial on $b_2,b_3$ and $b_5$ where $b_2b_3$ appears with coefficient $f_{125}-f_{134}$. If $f_{125}+f_{134} \neq 0$ then $wt_{P_{456}}(f) \geq (q-1)^2q$. If $f_{125}-f_{134} \neq 0$ then $wt_{P_{356}}(f) \geq (q-1)^2q$. If both $f_{125}-f_{134} \neq 0$ and $f_{125}+f_{134} \neq 0$ then $wt(f) \geq 2(q-1)^2q$. Since we are assuming $q$ is odd, then $wt(f) > wt_{P_{456}}(f) + wt_{P_{356}}(d) > q^3-q^2$. Therefore we may assume $f_{134} = f_{125}$.  Now we compare the weights on $P_{246}$ and $P_{236}$.

The matrices in $P_{246}$ are of the form $\begin{bmatrix}
0 & 0 & c_2 & 0 & c_3 & 1\\
-c_2 & 0 & 0 & 1 & 0 & 0\\
-c_3 & 1 & 0 & 0 & 0 & 0\\
\end{bmatrix}$ where $c_2,c_3 \in \F_q$. 
As in the case of the previous pivots, no term of the form $c_2^2$, $c_3^2$ appears.
On $P_{246}$ the function $det_{125}(M_W) = 
\begin{vmatrix}  
  0& 0 & c_3 \\
 -c_2 & 0 & 0 \\
 -c_3 & 1 & 0 \\ \end{vmatrix} = -c_2c_3$ and $det_{134}(M_W) = 
\begin{vmatrix}
  0& c_2 & 0 \\
 -c_2 & 0 & 1 \\
 -c_3 & 0 & 0 \\ \end{vmatrix} = -c_2c_3$. Since $f_{125} + f_{134} = 2 \neq 0$ reasoning as in the previous case $wt_{P_{246}}(f) \geq (q-1)^2$.
 
The matrices in  $P_{145}$ are of the form  $$\begin{bmatrix}
0 & 0 & e_2 & 0 & 1 & 0\\
0 & -e_2 & 0 & 1 & 0 & 0\\
1 & 0 & 0 & 0 & 0 & 0\\
\end{bmatrix}.$$ In this case $det_{125}(M_W) = 
\begin{vmatrix}  
  0& 0 & 1 \\
 0 & -e_2 & 0 \\
 1 & 1 & 0 \\ \end{vmatrix} = e_2$ and $det_{134}(M_W) = 
\begin{vmatrix}
  0& e_2 & 0 \\
 0 & 0 & 1 \\
 1 & 0 & 0 \\ \end{vmatrix} = e_2$. Since $e_2$ appears with $2$ as a coefficient $wt_{P_{145}}(f) = q-1$.
 
 Adding the different weights together we obtain:
 
 $$wt(f)\geq wt_{P_{456}}(f) + wt_{P_{246}}(f) + wt_{P_{145}}(f)\geq (q-1)^2q+(q-1)^2 + (q-1) = q^3-q^2.$$ 
  \end{proof}

Now we have determined that all nonzero $f\in \Dfts$ with no minors of the form $I = I^*$ in their support have weight at least $q^3-q^2.$ Now we assume it has a minor of the form $I = I^*$.

\begin{lemma}

Let $f \in \Dfts$ where $f \neq 0$. Suppose that $f_{123}$. Then $wt(f) \geq q^3-q^2$.
\end{lemma}
\begin{proof}
Let $f$ be as in the hypothesis of the lemma. From Lema \ref{lem:principal123} we may assume $f_{123} = 1$. First we shall assume $det_{124}$ appears in the support of $f$. We shall consider the weight of $f$ evaluated on $P_{456}$ and $P_{356}$. On any matrix $M_W \in P_{456}$ $det_{123}(M_W) = 0$ and $det_{124}(M_W) = a_2^2$.  However, on any matrix $M_W \in P_{356}$ $det_{123}(M_W) = b_2^2$ and $det_{124}(M_W) = 0$.
If both $f_{123}$ and $f_{124}$ are not zero, then $wt_{P_{456}}(f) \geq (q-2)q^2 $ and $wt_{P_{356}}(f) \geq (q-2)q^2 $. Therefore $wt(f) \geq 2(q-2)q^2 = 2q^3-4q^2 \geq q^3-q^2.$ Note that $f$ evaluated on $P_{124}$ and $P_{123}$ is nonzero, which implies $wt(f) > q^3-q^2.$

Now we assume $f_{124} = 0$. A similar argument on the evaluation of $f$ on $P_{356}$ implies that if $f_{135}, f_{236} \neq 0$ then $wt(f) > q^3-q^2$. Therefore we assume $f_{124} = f_{236} = f_{135} = 0$.

If $f_{134} \neq f_{125}$ then $a_3a_2$ appears in the evaluation of $f$ on $P_{456}$. This implies $wt_{P_{456}}(f) \geq (q-1)^2q$. Since $wt_{P_{356}}(f) \geq (q-2)q^2 $ and therefore $wt(f) \geq (q-1)^2q + (q-2)q^2 = 2q^3-4q^2+q > q^3-q^2$. Therefore we may assume $f_{134} = f_{125}$. Likewise we may assume $f_{234} = f_{126}$ and $f_{136} = f_{235}$.

If $f_{125} \neq 0$ we consider the following $g \in \Dfts$:
$$g = f(X_1 \vert X_2 + f_{125} X_4 \vert X_3 - f_{125} X_5 \vert X_4 \vert X_5 \vert X_6).$$
This will eliminate the term $det_{125}$ from the expression of $g$, but $g_{134} = f_{134}+f_{125}$. If $f_{125} \neq 0$, $g$ satisfies $g_{125} \neq g_{134}$ and therefore $wt(f) = wt(g) > q^3-q^2.$ 

Now we assume: $f_{123} = 1$ and $f_{124} = f_{125} = f_{126} = f_{134} = f_{135} = f_{136} = f_{234} = f_{235} = f_{236} = 0.$ This implies we may assume all determinants with two columns in the first three positions may be assumed to not appear in $supp(f)$. Now we consider the weights when determinants with one column on the first three positions.

We claim that if $f$ evaluates to a nonzero codeword on $P_{456}$, then the established bound on $P_{356}$ will imply $wt(f) > q^3-q^2$.
Note that $det_{346} = a_5$ and $det_{256} = -a_5$. Therefore if $f_{346} \neq f_{256}$, then $wt_{P_{456}}(f) \geq q^3-q^2$ and therefore $wt(f) > q^3-q^2.$ The same reasoning implies $f_{345} = f_{156}$  and  $f_{234} = f_{146}$. 

Note that $f$ has no functions which would evaluate to a $2\times 2$ minor on the pivot set $P_{356}$. If $f_{345}$ or $f_{346}$ are not zero, then $f$ evaluated on $P_{356}$ contains $b_3$ or $b_5$ and no other term containing $b_3$ or $b_5$ which implies $wt_{P_{356}} = q^3-q^2$ and therefore $wt(f)$ is too high.

If $456$ appears in the support, then we have a nonzero function on $P_{456}.$ Since the nonzero functions on $P_{456}$ have weight at least $(q-2)q^2$, this would imply the weight is too high.

We may now assume $f_{345} = f_{346} = f_{256} = f_{156} = 0.$ The only minors which may appear and not increase $wt(f)$ are $123$, $234$, $146$, $145$, $246$ and $356$ where $f_{123} = 1$ and $f_{234} = f_{146}.$

If $456$ appears in the support, then we have a nonzero function on $P_{456}.$ Since the nonzero functions on $P_{456}$ have weight at least $(q-2)q^2$, this would imply the weight is too high.

If $f_{234} \neq 0$ we consider the following $g \in \Dfts$:
$$g = f(X_1 - f_{234}X_4 \vert X_2 \vert X_3 + f_{234} X_6 \vert X_4 \vert X_5 \vert X_6).$$
In this case $g_{234} = 0$ but $g_{146} = 2 f_{234} \neq 0$. This implies $g$ evaluates to a nonzero function on both $P_{456}$ and $P_{356}$ and its weight is too high.

Therefore $f_{234} = f_{146} = 0$. We need only to consider the principal minors $123$, $145$, $246$ and $356$.

If only $123$ appears, then by swapping columns and considering the weight of $det_{456}$ we know that $wt(f) = q^3$.

If two or more $145$, $246$ or $356$ appear in the support of $P_2$, then the evaluation on $P_{356}$ is of the form $b_2^2 + f_{246}b_5^2 + f_{145}b_3^2 +f_{356}$. Without loss of generality we may permute columns and assume $f_{356} \neq 0$. This evaluation is a quadratic form on two or three variables. The estimates of the number of solutions to a quadratic equation from \cite{Lidl:1997} implies it has at most $q^2+q$ zeroes.  This improves the estimate of $wt_{P_{356}}(f)$ to $(q^2-q-1)q$. Using the simpler estimates of $wt_{P_{246}}(f) \leq (q-2)q$, $wt_{P_{145}}(f) \leq (q-2)$ and $wt_{P_{123}}(f) = 1$ we obtain that $wt(f) \leq wt_{P_{356}}(f) +wt_{P_{246}}(f)+wt_{P_{145}}(f)+wt_{P_{123}}(f) \geq q^3-2q-1 > q^3-q^2. $

Therefore only one of $145$, $236$ or $356$ can be nonzero. If $-f_{356}$ is not a square, then $wt_{P_{356}}(f) = q^3$ and $wt(f) > q^3-q^2$. If $-f_{356}$ is a square then $wt(f) = q^3-q^2.$  \end{proof}

This leads to our main result.

\begin{theorem}\label{thm:orthogonaldistanceqodd}
The minimum distance of the polar Orthogonal Grassmann code for $q$ odd is $d(\CO) = q^3-q^2$.
\end{theorem}

\subsection{even characteristic case}

In this subsection we assume $q$ is an even prime power. Recall that in even characteristic the minors evaluated on $\Oll$ satisfy $det_{I} = det_{I^*}$.

\begin{lemma}\label{lem:min codew-orthogonal-even}
$$wt(det_{456}(\X)) = q^3. $$
\end{lemma}
\begin{proof}

Let $f = det_{456}(\X)$. Recall that the matrices in $P_{456}$ are of the form $ \begin{bmatrix}
0 & a_2 & a_3 & 0 & 0 & 1\\
-a_2 & 0 & a_5 & 0 & 1 & 0\\
-a_3 & -a_5 & 0 & 1 & 0 & 0\\
\end{bmatrix}$. This implies $det_{456}(M_W) = 1 $ for all $M_W \in P_{456}$. Therefore $wt_{P_{456}}(f) = q^3$. The minor $456$ is zero on all other pivot sets. Therefore $wt(f) = wt_{P_{456}}(f)$\end{proof}



Because of the different behaviour of $ev(X^2)$ in $\F_q$ for $q=2$ and for $q > 2$, we split the proof in two subcases.
\subsubsection{$q=2$}
In this subsection we shall assume $q=2$. As in the previous case we shall determine the weight of $f \in \Dfts$ when evaluated on the sets $P_I$. First we shall assume that $f$ contains only minors of the form $I = I^*$ in its support.

\begin{lemma}

Let $f \in \Dfts$ where $f \neq 0$. Suppose that all minors in the support of $f$ satisfy $I = I^*$. Then $wt(f) \geq q^3$
\end{lemma}
\begin{proof}
Let $f$ be as in the hypothesis of the lemma. Since $f \neq 0$ and it contains a principal minor in its support, Lemma \ref{lem:principal123} and rescaling implies we may assume $f_{123} = 1$.

Recall that $P_{356}$ is of the form $  \begin{bmatrix}
0 & b_2 & 0 & b_3 & 0 & 1\\
-b_2 & 0 & 0 & b_5 & 1 & 0\\
-b_3 & -b_5 & 1 & 0 & 0 & 0\\
\end{bmatrix}$. The evaluation of $f$ on $M_W$ for $M_W \in P_{356}$ is $b_2^2 + f_{246}b_5^2 + f_{145}b_3^2 + f_{356}.$ Since $q = 2$, this evaluates to the same function as $b_2 + f_{246}b_5 + f_{145}b_3 + f_{356}.$ Therefore $wt_{P_{356}}(f) = (q-1)q^2$.

The matrices in $P_{246}$ are of the form $$\begin{bmatrix}
0 & 0 & c_2 & 0 & c_3 & 1\\
-c_2 & 0 & 0 & 1 & 0 & 0\\
-c_3 & 1 & 0 & 0 & 0 & 0\\
\end{bmatrix}.$$ 
The same argument as in the previous case yields $wt_{P_{356}}(f) = (q-1)q$.
As in the case of the previous pivots, no term of the form $c_2^2$, $c_3^2$ appears.
The matrices in  $P_{145}$ are of the form  $$\begin{bmatrix}
0 & 0 & e_2 & 0 & 1 & 0\\
0 & -e_2 & 0 & 1 & 0 & 0\\
1 & 0 & 0 & 0 & 0 & 0\\
\end{bmatrix}.$$ Therefore $wt_{P_{145}}(f) = (q-1)$.
 Note that $P_{123}$ is the set containing the matrix $$\begin{bmatrix}
0 & 0 & 1 & 0 & 0 & 0\\
0 & 1 & 0 & 0 & 0 & 0\\
1 & 0 & 0 & 0 & 0 & 0\\
\end{bmatrix}.$$ Therefore $wt_{P_{123}}(f) = 1$,
 Adding the different weights together we obtain:
 
 $$wt(f)\geq wt_{P_{456}}(f) + wt_{P_{246}}(f) + wt_{P_{145}}(f) + wt_{P_{123}}(P) \geq q^3.$$ 
  \end{proof}
  
  Because all minors were principal, no terms of the form $b_i$ nor $b_ib_j$, $b_i \neq b_j$ appeared. Now we compute the remaining case.

\begin{lemma}

Let $f \in \Dfts$ where $f \neq 0$. Suppose that there is a nonprincipal minor $I$ such that $I \in supp(f), I^* \not \in supp(f)$. Then $wt(f) \geq q^3$
\end{lemma}
\begin{proof}
Let $f$ be as in the hypothesis of the lemma. Without loss of generality we may assume $f_{125} = 1$ and $f_{234} = 0$. On $P_{456}$ and $P_{356}$ the function $det_{125}$ evaluates to $a_2a_3$ or $b_2b_3$ respectively. As this is the only minor which has the term $a_2a_3$ or $b_2b_3$ it follows that $wt_{P_456}(f), wt_{P_{356}}(f) \geq (q-1)^2q = 2$.

The function $det_{125}$ evaluates to $c_2c_5$ on $P_{246}$ but evaluates to $d_2d_3$ on $P_{236}$. This leads to the bound  $wt_{P_{246}}(f),wt_{P_{236}}(f) \geq 1$.

The function $det_{125}$ evaluates to $e_2$ on $P_{145}$ and evaluates to $e_2$ on $P_{135}$ This leads to the bound  $wt_{P_{145}}(f) \geq q-1=1$ and $wt_{P_{135}}(f) \geq q-1 = 1$
 Adding the different weights together we obtain:
 
 $$wt(f)\geq 2(2)+2(1)+2(1) = 8 = q^3.$$ \end{proof}

This leads to our main result.

\begin{theorem}\label{thm:orthogonaldistanceq2}
The minimum distance of the polar Orthogonal Grassmann code for $q=2$ is $d(\CO) = q^3$.
\end{theorem}
\subsubsection{$q$ even, $q>2$}

Now we compute the distance of polar Orthogonal Grassmann codes for $q$ even and $q \neq 2$. 
\begin{lemma}

Let $f \in \Dfts$ where $f \neq 0$. Suppose that no minors in the support of $f$ satisfy $I = I^*$. Then $wt(f) \geq q^3$
\end{lemma}
\begin{proof}
Let $f$ be as in the hypothesis of the lemma. After applying Lemma \ref{lem:principal125} we may assume $f_{125} = 1$ and $f_{134} = 0.$  Now we shall estimate the weight of the evaluation of $f$ by considering the contributions given by $det_{125}$ on $P_{456}$ and $P_{356}$.
 
 The matrices in $P_{456}$ are of the form $\begin{bmatrix}
0 & a_2 & a_3 & 0 & 0 & 1\\
-a_2 & 0 & a_5 & 0 & 1 & 0\\
-a_3 & -a_5 & 0 & 1 & 0 & 0\\
\end{bmatrix}$. When we evaluate $f$ on the matrices $M_W$ note that no term of the form $a_2^2, a_3^2 , a_5^2$ appears because we have no minors satisfying $I = I^*$.  Since $det_{125}(M_W) =  a_2a_3$ we estimate $wt_{P_{456}}(f) \geq (q-1)^2q$

Recall that $P_{356}$ is of the form $  \begin{bmatrix}
0 & b_2 & 0 & b_3 & 0 & 1\\
-b_2 & 0 & 0 & b_5 & 1 & 0\\
-b_3 & -b_5 & 1 & 0 & 0 & 0\\
\end{bmatrix}$. As in the case of $P_{456}$ when we evaluate $f$ on the matrices $M_W$ note that no term of the form $b_2^2, b_3^2 , b_5^2$ appears because we have no minors satisfying $I = I^*$. In this case $det_{125}(M_W) =  b_2b_3$  and therefore we estimate $wt_{P_{356}}(f) \geq (q-1)^2q$. Now we compare the weights on $P_{246}$ and $P_{236}$. The matrices in $P_{246}$ are of the form $\begin{bmatrix}
0 & 0 & c_2 & 0 & c_3 & 1\\
-c_2 & 0 & 0 & 1 & 0 & 0\\
-c_3 & 1 & 0 & 0 & 0 & 0\\
\end{bmatrix}$. As in the case of the previous pivots, no term of the form $c_2^2$, $c_3^2$ appears. The determinant $det_{125}(M_W) =  c_2c_3$ we estimate $wt_{P_{246}}(f) \geq (q-1)^2$. The matrices in $P_{236}$ are of the form $\begin{bmatrix}
0 & 0 & 0 & c_2 & c_3 & 1\\
-c_2 & 0 & 1 & 0 & 0 & 0\\
-c_3 & 1 & 0 & 0 & 0 & 0\\
\end{bmatrix}$. As in the case of the previous pivots, no term of the form $c_2^2$, $c_3^2$ appears. The determinant $det_{125}(M_W) =  c_2c_3$ we estimate $wt_{P_{236}}(f) \geq (q-1)^2$

The matrices in  $P_{145}$ are of the form  $\begin{bmatrix}
0 & 0 & e_2 & 0 & 1 & 0\\
0 & -e_2 & 0 & 1 & 0 & 0\\
1 & 0 & 0 & 0 & 0 & 0\\
\end{bmatrix}$. In this case $det_{125}(M_W) = e_2$ and $wt_{P_{145}}(f) \geq (q-1)$.  The matrices in  $P_{135}$ are of the form  $\begin{bmatrix}
0 & 0 & 0 & e_2 & 1 & 0\\
0 & -e_2 & 1 & 0 & 0 & 0\\
1 & 0 & 0 & 0 & 0 & 0\\
\end{bmatrix}$. In this case $det_{125}(M_W) = e_2$ and $wt_{P_{135}}(f) \geq (q-1)$
 
 Adding the different weights together we obtain:
 
 $$wt(f)\geq 2(q-1)^2q + 2(q-1)^2 +2(q-1) = 2q^3-2q^2 > q^3.$$ \end{proof}
 
 Now we solve the remaining case for $q > 2$, which is when $f$ has a principal minor in its support.
 \begin{lemma}

Let $f \in \Dfts$ where $f \neq 0$. Suppose there is a minor $I \in supp(f)$ such that $I = I^*$. Then $wt(f) \geq q^3$.
\end{lemma}
\begin{proof}
Let $f$ be as in the hypothesis of the lemma. Applying Lemma \ref{lem:principal123} and rescaling, we may assume $f_{123} = 1$.

Recall that $P_{356}$ is of the form $  \begin{bmatrix}
0 & b_2 & 0 & b_3 & 0 & 1\\
-b_2 & 0 & 0 & b_5 & 1 & 0\\
-b_3 & -b_5 & 1 & 0 & 0 & 0\\
\end{bmatrix}$. The evaluation of $f$ on $M_W$ for $M_W \in P_{356}$ contains the monomial $b_2^2$ Therefore $wt_{P_{356}}(f) = (q-2)q^2$.

The matrices in $P_{246}$ are of the form $\begin{bmatrix}
0 & 0 & c_2 & 0 & c_3 & 1\\
-c_2 & 0 & 0 & 1 & 0 & 0\\
-c_3 & 1 & 0 & 0 & 0 & 0\\
\end{bmatrix}$. A similar argument as in the previous case yields $wt_{P_{246}}(f) = (q-2)q$.

The matrices in  $P_{145}$ are of the form  $\begin{bmatrix}
0 & 0 & e_2 & 0 & 1 & 0\\
0 & -e_2 & 0 & 1 & 0 & 0\\
1 & 0 & 0 & 0 & 0 & 0\\
\end{bmatrix}$. Therefore $wt_{P_{145}}(f) = (q-2)$.
 Note that $P_{123}$ is the set containing the matrix $\begin{bmatrix}
0 & 0 & 1 & 0 & 0 & 0\\
0 & 1 & 0 & 0 & 0 & 0\\
1 & 0 & 0 & 0 & 0 & 0\\
\end{bmatrix}$. Therefore $wt_{P_{123}}(f) = 1$,
 Adding the different weights together we obtain:
 
 $$wt(f)\geq wt_{P_{456}}(f) + wt_{P_{246}}(f) + wt_{P_{145}}(f) + wt_{P_{123}}(P)$$
 $$wt(f) \geq (q-2)q^2+(q-2)q +(q-2) +1= q^3-q^2-q-1.$$ We shall use this as our basic bound. Now we shall assume if certain coefficients are nonzero, then $wt(f)$ is too large.
 
 If $f$ has a nonprincipal minor which evaluates to $0$, for example $f_{125} \neq 0 $ and $f_{134} = 0$. Then $f$ evaluates to a nonzero function on $P_{456}$.  (Recall that $det_{125}+det_{134}$ evaluates to $0$). In this case $wt_{P_{456}}(f) \geq (q-1)^2q$.  The bound on $wt(f)$ becomes   $$wt(f)\geq q^3-q^2-q-1 + q^3-2q^2+q = 2q^3-3q^2-1.$$ Since $q \geq 4$ this bound is larger than $q^3$.
 
 Now we shall assume $f$ has no nonprincipal minors. On the different pivot sets the principal minors evaluate to either $0$, $1$, or $T^2$, for some variable $T.$ Since we are in characteristic $2$, and the evaluation of $f$ only has squares of variables, the weight bound is improved to  $$wt(f)\geq  (q-1)q^2+(q-1)q + (q-1) +1= q^3.$$ 
 \end{proof}

 Our main result is thus.
 
 \begin{theorem}
 
 The minimum distance of the code $\CO$ is $$
d(\CO)=
\begin{cases}
q^3, q \makebox{ even }\\
q^3-q^2, q \makebox{ odd }\\
\end{cases}
$$
 \end{theorem}

 We also check \cite{Grassl:codetables} and confirm that for $q=2$ the code $\CO$ is a $[30,14,8]_2$ code which is optimal.

\section{Conclusion}
In this paper we have determined the minimum distance of the polar orthogonal Grassmann codes $\CO$ using elementary techniques for even and odd characteristic. We hope to apply these techniques to other parameters of polar Grassmannians.

\section*{Acknowledgments} The authors thank Bart De Bruyn, Ilaria Cardinali, Luca Giuzzi and Sudhir Ghorpade for their thoughtful discussions and marvelous suggestions to improve this manuscript. This research is supported by NSF-DMS REU 1852171: REU Site: Combinatorics, Probability, and Algebraic Coding Theory and NSF-HRD 2008186: Louis Stokes STEM Pathways and Research Alliance: Puerto Rico-LSAMP - Expanding Opportunities for Underrepresented College Students (2020-2025)

\end{document}